\documentclass{article}

\usepackage{graphicx}
\usepackage{amssymb,amsmath, amsthm}
\usepackage[pagewise]{lineno}

\newtheorem{theo}{Theorem}
\newtheorem{coro}{Corollary}[section]
\newtheorem{lemm}[coro]{Lemma}
\newtheorem{prop}[coro]{Proposition}
\newtheorem{rema}[coro]{Remark}
\newtheorem{defi}[coro]{Definition}
\newtheorem{ques}{Question}

\newcommand{\uind}{\mathrm{u}\mbox{-}\mathrm{ind}}

\title{Super exponential divergence of periodic 
points for $C^1$-generic partially hyperbolic 
homoclinic classes}

\author{Xiaolong Li and Katsutoshi Shinohara}

\begin{document}

\maketitle

\begin{abstract}
A diffeomorphism $f$ is called super exponential divergent if for every $r>1$, the lower limit of $\#\mbox{Per}_n(f)/r^n$ diverges to infinity as $n$ tends to infinity, where $\mbox{Per}_n(f)$ is the set of all periodic points of $f$ with period $n$. This property is stronger than the usual super exponential growth of the number of periodic points. We show that for a three dimensional manifold $M$, there exists an open subset $\mathcal{O}$ of $\mbox{Diff}^1(M)$ such that diffeomorphisms with super exponential divergent property form a dense subset of $\mathcal{O}$ in the $C^1$-topology. A relevant result of non super exponential divergence for diffeomorphisms in a locally generic subset of $\mbox{Diff}^r(M)\ (1\le r\le \infty)$ is also shown. 
\end{abstract}
\footnote{XL: lixl@hust.edu.cn\\
Graduate School of Mathematics and Statistics,\\
Huazhong University of Science and Technology, Luoyu Road 1037, Wuhan, China\\
}
\footnote{KS: ka.shinohara@r.hit-u.ac.jp\\
Graduate School of Business Administration,\\
 Hitotsubashi University, 2-1 Naka, Kunitachi, Tokyo, Japan
 }

{\footnotesize{2020 \emph{Mathematics Subject Classification}:  37C20, 37C25, 37C29, 37D30}

}


\section{Introduction}

\subsection{Backgrounds}

The investigation of the growth of the number of 
periodic points for dynamical systems is a fundamental 
problem. For uniformly hyperbolic systems, we know that 
the growth of the number of periodic points cannot be faster 
than some exponential functions. Then a natural question is that 
what happens for systems which fail to be uniformly hyperbolic 
in a robust fashion. 

A fundamental result is given by Artin and Mazur, which 
asserts that for a dense subset of $C^r$ maps of a compact manifold into itself with the uniform $C^r$ topology, the number of 
isolated periodic points grows at most exponentially \cite{AM}. 
Meanwhile, there are some results for locally generic maps. For instance, Bonatti, D\'iaz and Fisher shows that generically in $\mbox{Diff}^1(M)$, if a homoclinic class contains periodic points of different indices, then it exhibits super exponential growth of number of periodic points \cite{BDF}; For certain semi-group actions on the interval, Asaoka, Shinohara and Turaev construct $C^r$ $(r\ge 1)$ open set in which $C^r$ generic maps exhibit super exponential growth of number of periodic points \cite{AST}; For $C^r$ diffeomorphisms of compact smooth manifolds, they also construct local $C^r$ generic subset with fast growth of number of periodic points under certain conditions about the signatures of non-linearities and Schwarzian derivatives of the transition maps \cite{AST2}; Berger shows that for $2\le r\le \infty$ and manifolds of dimension greater than 1, there exists open set $\mathcal{O}\subset\mbox{Diff}^r(M)$ in which $C^r$ generic $f$ displays a fast growth of the number of periodic points \cite{Be}. Thus one may consider the difference of the growth as 
a probe of the degree of the non-hyperbolicity which the system 
exhibits.

Let us be more precise. Given a set $X$ and 
a map $f:X \to X$, we say that $x \in X$ is a periodic 
point of period $n$ (where $n \geq 1$) if $f^n(x) = x$ and
$n$ is the least positive integer for which this equality holds. In particular, $x$ is called a fixed point of $f(x)=x$. 
We denote the set of periodic points of period $n$ of $f$
by $\mathrm{Per}_n(f)$.

For the investigation of the number of periodic points, we mainly focus on the {\it ratio} of $\#\mathrm{Per}_n(f)$ to $r^n\ (r>1)$.
We customarily consider upper limit ($\limsup$) 
of the ratio when $n$ goes to infinity. $f$ is called 
\emph{super-exponential} if for every $r >1$
the sequence $r^{-n}\#\mathrm{Per}_n(f)$ has 
upper limit equals to $+\infty$. One motivation of this definition is that the ``rate of exponential growth" of the number of periodic points comes from the 
investigation of the convergence radius of dynamical zeta function.
The positivity of the convergence 
radius is equivalent to the finitude 
of the upper limit of the ratio. In this case, at least a subsequence of $\#\mbox{Per}_n(f)$ grows exponentially fast, which implies that the dynamics exhibits relatively complicated behaviour. On the other hand, the cases of super exponentially fast growth also appear very often, as aforementioned papers indicated.

Meanwhile, as a measure of non-uniform hyperbolicity 
it is interesting to ask what happens for the lower limit ($\liminf$) of the 
ratio. Indeed, the lower limit provides us more information about the number of $n$-periodic points for {\it every} sufficiently large $n$. It is not easy to construct 
an diffeomorphism around which, maps whose lower limit of the ratio divergent to $+\infty$ exist persistently (for instance, in a dense or residual subset of a neighbourhood of the initial diffeomorphism). In this paper, we provide such an example.

We say that $f$ is \emph{super-exponentially divergent}
if for every $r >1$
the sequence $r^{-n}\#\mathrm{Per}_n(f)$ has 
lower limit equals to $+\infty$. Indeed, this is equivalent 
to say that the limit exists and it is equal to $+\infty$. Let $\mbox{Diff}^1(M)$ denote the space of $C^1$ diffeomorphisms of a manifold $M$, endowed with the $C^1$ topology. 

\begin{theo}\label{thm1}
There exists a three dimensional closed  
manifold $M$ such that the following holds:
There exist an non-empty 
open set $\mathcal{O} \subset \mathrm{Diff}^1(M)$
 and a dense subset $\mathcal{D}$ of $\mathcal{O}$ such that every diffeomorphism 
in $\mathcal{D}$ is super-exponentially divergent. 
\end{theo}

Let us make some comments. Our construction is based on 
the bifurcation of heterodimensional cycles. Since heterodimensional 
cycles exist only for manifolds whose dimension is greater 
than two, we are not sure a similar result holds for surface diffeomorphisms. 
We gave this result for $C^1$-regularity. As we will see, 
our technique heavily depends on the nature of $C^1$-distance.
Thus the $C^r$-case for $r >1$ is open. 

In the following, we give the description of the open set 
$\mathcal{O}$.

\subsection{Results}

Let $M$ be a closed $n$-dimentional Riemannian manifold.
We fix a Riemaniann metric $\| \cdot \|$ on $TM$ and 
a metric $d$ on $M$. Denote by $\mathrm{Diff}^1(M)$ the space of $C^1$ diffeomorphisms of $M$ endowed with the 
$C^1$ topology. 
We also fix a metric $\mathrm{dist}(f,g)$ for every 
pair of $f,g\in \mathrm{Diff}^1(M)$ 
which is compatible with the $C^1$ topology.

Let us recall some notion for non-uniformly 
hyperbolic systems, following \cite{BDU}. 
For further information, 
see for instance \cite{BDV}.
Let $f\in \mathrm{Diff}^1(M)$ and 
$\Lambda \subset M$ be an $f$-invariant set, 
that is, $f(\Lambda) = \Lambda$ holds. 
Let $E, F$ be subbundles of $TM|_{\Lambda}$ which are
invariant under $Df$ respectively.
$E_x \cap F_x = \{ 0 \}$ for every $x \in \Lambda$.  
We say that $E \oplus F$ is a \emph{dominated splitting}
if there exists a positive real number $\alpha$
strictly smaller than 1 such that 
for every $x \in \Lambda$ we have 
$\|Df|_{E_x}\| \cdot \|Df^{-1}|_{F_{f(x)}}\| < \alpha$,
where $\|Df|_{E_x}\|$ denotes the operator norm of 
$Df|_{E_x}$ with respect to the Riemannian metric.

We say that $\Lambda$
is \emph{strongly partially hyperbolic} 
if there is a splitting 
$TM|_{\Lambda} = E^s \oplus E^c \oplus E^u $ such that
$E^s \oplus (E^c \oplus E^u)$ and
$(E^s \oplus E^c) \oplus E^u$ are dominated splittings
with $\dim E^c =1, \dim E^s \geq 1$ and $\dim E^u \geq 1$, $E^s$ is uniformly contracting and $E^u$ is uniformly expanding.
We say that $Df$ is orientation preserving if 
if $E^s, E^c, E^u$ are all orientable and $Df$ preserves 
these orientations.

Suppose we have a pair of hyperbolic periodic 
points $P_1, P_2 \in \Lambda$.  
We say that they are \emph{adapted} if 
$\uind(P_1) = \dim(E^u) +1$ and  
$\uind(P_2) = \dim(E^u)$, where
$\uind(P)$ denotes the dimension of the unstable subspace of a hyperbolic periodic point $P$. 

We say that $f$ is \emph{transitive} (on $M$) 
if there is an orbit which is dense in $M$, that is, 
if there exists $x \in M$ such that 
$\{f^{n}(x)\}_{n \in \mathbb{Z}}$ is dense in $M$.
A diffeomorphism $f$ is $C^r$-\emph{robustly transitive}
if there is an open neighbourhood $\mathcal{U}$ of 
$f$ in $\mathrm{Diff}^r(M)$ equipped with the 
$C^r$ topology such that every 
$g \in \mathcal{U}$ is transitive.  

The following is our first result.
\begin{theo}\label{theo:trans}
Let $M$ be a three dimensional closed manifold and 
$f$ be a $C^1$-robustly 
transitive diffeomoprhism for which the entire 
manifold $M$ is a strongly partially hyperbolic 
set. 
Suppose that $f$ has two hyperbolic 
fixed points $P_1$ and $P_2$ having 
$u$-indices $2$ and $1$ respectively and  
$Df$ preserves the orientations of the strongly 
partially hyperbolic splitting over $M$. 
Then there exist a $C^1$-neighbourhood $\mathcal{U}$
 of $f$ in $\mathrm{Diff}^1(M)$ and a dense subset 
 $\mathcal{D}$ of $\mathcal{U}$ satisfying the following:
 Every $g \in \mathcal{D}$ is super-exponentially 
 divergent.
\end{theo}

Theorem~\ref{theo:trans} is a consequences of 
general perturbation results together with 
the following analytic result.

First, let us state the analytic result.
\begin{theo}\label{theo:main}
Let $M$ be a three dimensional closed manifold.
Suppose $f$ satisfies the following:
\begin{itemize}
\item[(T1)] (Codimension-1 property) There are hyperbolic fixed points $P_1$ and $P_2$
of $f$ with 
$\uind(P_1)=2$ and  
$\uind(P_2)=1$.
\item[(T2)] (Simplicity property) The weakest unstable eigenvalue of $P_1$
and the weakest stable eigenvalue of $P_2$ 
are both real, positive and have multiplicity one.
\item[(T3)] (Existence of a strong heteroclinic intersection)
Let $W^{ss}(P_2)$ denote the 
stong stable manifold of $P_2$ corresponding 
to the strong stable eigenvalue of $Df(P_2)$.
Then $W^{u}(P_1) \cap W^{ss}(P_2) \neq \emptyset$.
\item[(T4)] (Existence of a quasi-transverse intersection)
$W^u(P_2) \cap W^s(P_1) \neq \emptyset$.
\end{itemize}
Then, there is a diffeomorphism $g$ 
which is arbitrarily $C^1$ close to $f$ such that
for every $r >1$ 
we have 
\[\lim\limits_{n\to \infty}\dfrac{\# \mathrm{Per}_n(g)}{r^n}=+\infty.\]
Notice that this implies $\liminf_{n \to \infty} \dfrac{\# \mathrm{Per}_n(g)}{r^n}=+\infty$.
\end{theo}

One important condition in the assumptions of
Theorem~\ref{theo:main} 
is that we assume 
$P_1, P_2$ are \emph{fixed} points of $f$.
In this article, we are interested in the behavior of 
lower limit of the number of periodic points. 
In many situations, the difference whether the periodic 
orbit we are interested in has non-trivial period or not 
can be overcome by taking power of the dynamics. 
On the other hand, as we will see later, this strategy does 
not work in a simple way for the investigation of the lower limit. 
The investigation of to what extent we can relax this 
fixed point assumption would be an interesting topic, but
we will not pursue this problem in this paper.

The following perturbation result tells us that 
for systems which satisfy the hypothesis of 
Theorem~\ref{theo:trans}, we can obtain the 
assumptions of Theorem~\ref{theo:main} 
up to an arbitrarily small $C^1$ perturbation. 

\begin{prop}\label{prop:pert}
Let $f$ be a $C^1$-robustly transitive 
diffeomorphism on a smooth compact 
three dimensional manifold such that the entire 
manifold is strongly partially hyperbolic. 
Assume that there are two hyperbolic fixed 
points $P_1$ and $P_2$ whose indices 
are $2$ and $1$ respectively and $Df$ preserves 
the orientation. Then, there exists $g\in\mathrm{Diff}^1(M)$
arbitrarily $C^1$ close to $f$ such that 
$g$ satisfies the hypotheses
(T1-4) of
Theorem~\ref{theo:main}.
\end{prop}

The proof of Proposition~\ref{prop:pert} will be given 
in Section~3.

Let us give a ``local version'' of Theorem~\ref{theo:trans}.
In Theorem~\ref{theo:trans}, we stated the result under 
the condition that the diffeomorphism is robustly transitive. 
While this condition is easy to understand, 
it is not the essential one which we need 
to reach the conclusion.
Below we give Theorem~\ref{theo:hclass}, 
in which the assumption for the super exponential 
divergence is stated in terms of homoclinic classes
and this statement makes it easier to grasp about the 
mechanism of the result.  

Let us recall the notion of homoclinic classes. 
A \emph{homoclinic class} of a hyperbolic periodic 
saddle $P$ of a diffeomorphism $f$,
denoted by $H(P; f)$ or $H(P)$, 
is defined to be the closure of transversal
intersections of the stable and unstable manifolds of $P$. 
We can equivalently define $H(P)$ as the closure of all 
hyperbolic periodic saddles $Q$ homoclinically related to $P$
 (i.e. the stable manifold of $P$ transversally intersects the 
 unstable manifolds of $Q$ and vice versa). 
Homoclinic classes are always invariant and transitive, 
but not necessarily hyperbolic in general (see for instance \cite{ABCDW}).

\begin{theo} \label{theo:hclass}
Let $M$ be a three dimensional closed manifold. Suppose $f\in \mathrm{Diff}^1(M)$ satisfies the following:
\begin{itemize}
\item There are hyperbolic fixed points $P_1$ and $P_2$ of $f$ contained in the same homoclinic class $H(P_1)$ with
$\uind(P_1)=2$ and $\uind(P_2)=1$;
\item The weakest unstable eigenvalue of $P_1$ and the weakest stable eigenvalue of $P_2$ are both real, positive and have multiplicity one;
\item $W^u(P_1)$ intersects $W^{ss}(P_2)$ transversally.
\end{itemize}
Then, there exists an arbitrarily small $C^1$-perturbation $g$ of $f$ such that $g$ is super exponential divergent.  
\end{theo}

The proof of Theorem~\ref{theo:hclass} will be given 
in Section~\ref{s:creSH}.
 
 Let us see the strategy of the proof of Theorem~\ref{theo:main}.
Given a heterodimensional cycle associated to two hyperbolic fixed points $P_1$ and $P_2$ of $f$, in other words, $P_1$ and $P_2$ are index adapted with $W^u(P_1)\cap W^s (P_2)\not=\emptyset$ and $W^u(P_2)\cap W^s(P_1)\not=\emptyset$. Following the 
argument in 
\cite{BD}, 
we linearize the dynamics around $P_1$ and $P_2$.
We also make the transition map along the heteroclinic points $Q_1$ and 
$Q_2$ affine, which are compatible with the 
linear maps around $P_1$ and $P_2$. 
These dynamics are called a \emph{simple cycle}. 
A direct calculation shows that
we can find a sequence of periodic points 
with weak center Lyapunov exponents. 
Thus, by exploiting the flexibility of the $C^1$-topology, 
we can increase the number of periodic points 
as much as we want by perturbing 
in the center direction. 

This is the strategy of the proof 
of \cite{BDF}. 
In our problem, we furthermore need to investigate
the \emph{frequency of the period of the weak 
periodic points}. To be more precise, we need to 
confirm the occurrence of periods of weak 
periodic points for \emph{every} sufficiently large integer. 
By investigating the above calculation carefully, we can 
observe that, 
under certain quantitative assumption on 
the characteristics of the simple cycle, 
the periods exhaust all sufficiently large integers eventually by an arbitrarily small $C^1$ perturbation. In our proof, an additional hypothesis on the simple cycle is needed. We call our simple cycle as \emph{SH-simple cycle}
(where SH stands for ``Strongly Heteroclinic''), meaning that in addition, the unstable manifold of $P_1$ intersects the strong stable manifold of $P_2$.   

\medskip

A natural question regarding Theorem~\ref{thm1}
is if one could replace ``dense'' to some stronger 
condition such as residual or open and dense. 
For instance, one might wonder the following:
\begin{ques}
Does there exist an open subset $\mathcal{U}$ of $\mbox{Diff}^1(M)$ such that generically in $\mathcal{U}$, diffeomorphisms are super exponential divergent?
\end{ques}
While we do not have an answer, 
in Section 6 we will prove one result about
the lower limit of the number of periodic points 
valid for diffeomorphisms 
in a residual subset of $\mbox{Diff}^r(M)$ where 
$1\le r\le \infty$, based on the argument of 
Kaloshin \cite{Ka}. 
\begin{theo}\label{theo:resi}
Given $1 \leq s \leq +\infty$ and 
a super-exponential sequence $(a_n)$,
there exists an residual set 
$\mathcal{R}$ of $\mathrm{Diff}^s(M)$
such that the following holds: for every 
$f \in \mathcal{R}$, we have $\liminf_{n \to \infty}\#\mathrm{Per}_n(f) / a_n \to 0$. 
\end{theo}
Thus we cannot extend Theorem~\ref{thm1}
in a straightforward way and 
this shows the significance of Theorem~\ref{thm1}. 
Notice that Theorem~\ref{theo:resi}
does not answer Question~1, because 
there is no ``slowest'' super exponentially 
increasing sequence.
  
\medskip

This paper will be organized as follows. In Section 2, after giving some basic definitions and notations on SH-simple cycles, we provide the proof of Theorem~\ref{theo:main} by assuming 
several results which will be proved in the following 
sections. 
In Section 3, we will discuss the proof of Theorem~\ref{theo:hclass} and
how to prove Theorem~\ref{theo:trans} from Theorem~\ref{theo:main}. 
Section 4 is devoted to the 
proof of Proposition~\ref{prop:SHsim},
a perturbation result for obtaining SH-simple cycles. 
In Section 5, we prove Proposition~\ref{prop:ana}, 
an analytic result for SH-simple cycles is shown. 
In Section 6, by using a theorem of \cite{Ka}, we prove 
Theorem~\ref{theo:resi}, a generic result of super exponential divergence with respect to a given speed.  

\bigskip

\noindent
{\bf Acknowledgements.} \,\, This paper has been supported by the the JSPS KAKENHI Grant Number 18K03357. XL is supported by the Youth Program of National Natural Science Foundation of China (11701199) and the Fundamental Research Funds for the Central Universities, HUST: 2017 KFYXJJ095.
KS thank the warm hospitality of 
Huazhong University of Science and Technology.
We thank Masayuki Asaoka and Ken-ichiro Yamamoto
for their suggestions and comments.


\section{Preliminaries and Strategies}

In this section, we prepare some definitions and cite
known results which are used throughout this paper.

Then, we state some propositions which will be used 
for the proof of Theorem~\ref{theo:main}. Finally, 
assuming these propositions we give the 
proof of Theorem~\ref{theo:main}.

\subsection{SH-simple cycles}

Let us give the definition of simple cycles. 
In this paper, it is more convenient if we define a wider 
class of simple cycles, which we call \emph{SH-simple} cycles.
Let us give the definition of it.

Let $\mathbb{D}_r^n := \{ x \in \mathbb{R}^n \mid \|x\| < r\}$, 
where $\| \cdot \|$ denotes the Euclidian norm and $r$ is some positive real number. 
Let $d_s, d_c, d_u$ be positive integers and put $d = d_s + d_c + d_u$
and $\mathbf{d} = (d_s, d_c, d_u)$.
A subset 
$\mathbb{D}^{\mathbf{d}} = \mathbb{D}_{r_s}^{d_s} \times \mathbb{D}_{r_c}^{d_c} 
\times \mathbb{D}_{r_u}^{d_u}$ of $\mathbb{R}^d$
is called a \emph{polydisc} of $\mathbb{R}^d$ of index $\mathbf{d}$.
We call the numbers $\mathbf{d}$ and
$(r_s, r_c, r_u)$ the \emph{index} and 
the \emph{size} of the 
polydisc $\mathbb{D}^{\mathbf{d}}$ respectively.
In the following, we only consider the case $d_c =1$. Thus 
we have $\mathbf{d} = (d_s, 1, d_u)$.

We first prepare two definitions, which describe 
the local dynamics around the fixed point and the 
transition dynamics near the heteroclinic point 
of SH-cycles respectively.

\begin{defi}
Let $f\in\mathrm{Diff}^1(M)$ and $P$ be its fixed point.
We say that a coordinate neighbourhood $(U, \phi)$ around $P$ is an 
\emph{linearized neighbourhood of index $\mathbf{d}$}
if the followings hold:
\begin{itemize}
\item $\phi(U) \subset \mathbb{R}^{\mathbf{d}}$ is a polydisc of index $\mathbf{d}$,
that is, $\phi(U) = \mathbb{D}_{r_s}^{d_s} \times \mathbb{D}_{r_c}^{1} 
\times \mathbb{D}_{r_u}^{d_u}$ for some 
size $(r_s, r_c, r_u)$.
\item There exist a contracting linear map $\Lambda: \mathbb{R}^{d_s} \to \mathbb{R}^{d_s}$,
a linear map $t: \mathbb{R} \to \mathbb{R}$ and 
an expanding linear map $M: \mathbb{R}^{d_u} \to \mathbb{R}^{d_u}$
such that the following holds:
For every $x \in U$, if 
$\phi(x) = (x_s, x_c, x_u)$ satisfies 
$t(x_c) \in \mathbb{D}_{r_c}^1$ and $M(x_u) \in \mathbb{D}_{r_u}^{d_u}$, 
then we have 
\[\phi(f(x)) = (\Lambda(x_s), t(x_c), M(x_u)).\]
We call the linear map $(x_s, x_c, x_u) \mapsto (\Lambda(x_s), t(x_c), M(x_u))$
the \emph{linearization} of $f$ near $P$.
We also call the set of the points $x$ satisfying
$t(x_c) \in \mathbb{D}_{r_c}^1$ and $M(x_u) \in \mathbb{D}_{r_u}^{d_u}$
the linearized region.
\end{itemize}
\end{defi}

\begin{defi}
Let $f\in\mathrm{Diff}^1(M)$, $P_1$ and $P_2$ be 
its adapted hyperbolic fixed points.
Suppose that $P_i$ have linearized coordinate neighbourhoods $(\phi_i, U_i)$
for $i=1$ and $2$ with linearization 
$(x_s, x_c, x_u) \mapsto (\Lambda_i(x_s), t_i(x_c), M_i(x_u))$ and assume that 
there exists a point $Q \in W^u(P_1) \cap W^s(P_2) \cap U_1$.
Let $\sigma >0$ be the least positive integer such that $f^{\sigma}(Q) \in U_2$ holds. 
We say that $Q$ is \emph{an adapted transition point} with respect to 
$(U_1, \phi_1)$
and
$(U_2, \phi_2)$
if 
there exist positive real numbers $\kappa_s, \kappa_c$ and $\kappa_u$
such that the followings hold:
\begin{itemize}
\item There exists a neighbourhood $K \subset U_1$ of $Q$ 
such that
$\phi_1(K)$ has the form 
$\phi_1(Q)+ \mathbb{D}_{\kappa_s}^{d_s} \times \mathbb{D}_{\kappa_c}^{1} \times \mathbb{D}_{\kappa_u}^{d_u}$,
that is, $\phi_1(K)$ is a polydisc centered at $\phi(Q)$
of size $(\kappa_s, \kappa_c, \kappa_u)$.
\item Furthermore, $f^i(K) \cap (U_1 \cup U_2) = \emptyset$ for every $i =1, \ldots, \sigma -1$ and 
$f^{\sigma}(K) \subset U_2$. 
\item 
There exist three linear maps $\tilde{\Lambda}: \mathbb{R}^{d_s} \to \mathbb{R}^{d_s}$,
$\tilde{t}: \mathbb{R} \to \mathbb{R}$ and 
$\tilde{M}: \mathbb{R}^{d_u} \to \mathbb{R}^{d_u}$
such that the following holds:
For every $(X, Y, Z) \in \mathbb{R}^{\mathbf{d}}$ such that 
$\phi_1(Q)+(X, Y, Z) \in \phi_1(K)$ holds, we have 
\[
\phi_2 \circ  f^{\sigma} \circ \phi_1^{-1}(\phi_1(Q)+(X, Y, Z)) = 
\phi_2(f^{\sigma}(Q))+(\tilde{\Lambda}(X), \tilde{t}(Y), \tilde{M}(Z)).
\]
\end{itemize}
We call $\tilde{t}$ the \emph{center multiplier} of the transition map $f^{\sigma}$
and $K$ the \emph{transition region}. 
\end{defi}

Given two hyperbolic fixed points $P_1$ and $P_2$ of 
different indices, 
we say that they form a \emph{heterodimensional cycle} 
if the stable manifold of $P_1$ intersects the unstable manifold of $P_2$ and vice versa. Heterodimensional cycles are realized as one of the typical mechanisms which causes the 
non-hyperbolicity of the dynamics.   

Now we are ready to state the definition of SH-simple cycles.

\begin{defi}
Let $f \in \mathrm{Diff}^1(M)$. 
Suppose there is a heterodimensional cycle associated to $P_1$, $P_2$ and 
heteroclinic points 
$Q_1 \in W^u(P_1) \cap W^s(P_2)$ and 
$Q_2 \in W^u(P_2) \cap W^s(P_1)$. 
We say that the heterodimensional cycle is \emph{SH-simple} if the followings hold (see Figure.~\ref{fig:SH}):
\begin{itemize}
\item There are linearized coordinates $(U_i, \phi_i)$ around $P_i$ for $i=1, 2$.
\item $Q_i \in U_i$ and 
they are adapted heteroclinic points associated to 
$(U_i, \phi_i)$ and $(U_{i+1}, \phi_{i+1})$ 
with transition maps $f^{\sigma_i}$ for $i =1, 2$, 
where we set 
$(U_{3}, \phi_{3})= (U_{1}, \phi_{1})$.
\item 
The $\mathbb{R}^{d_u}$-coordinate of $\phi_1(Q_1)$ is 
$\boldsymbol{0}^{d_u}$.
\item 
The $\mathbb{R}^{d_c}$-coordinate of 
$\phi_2(f^{\sigma_1}(Q_1))$ is $0$.
\item The center multipliers of the transition maps $f^{\sigma_i}\ (i=1,2)$ equal to 1.
\end{itemize}
\end{defi}

\begin{figure}[t]
\begin{center}
\includegraphics[width=12cm]{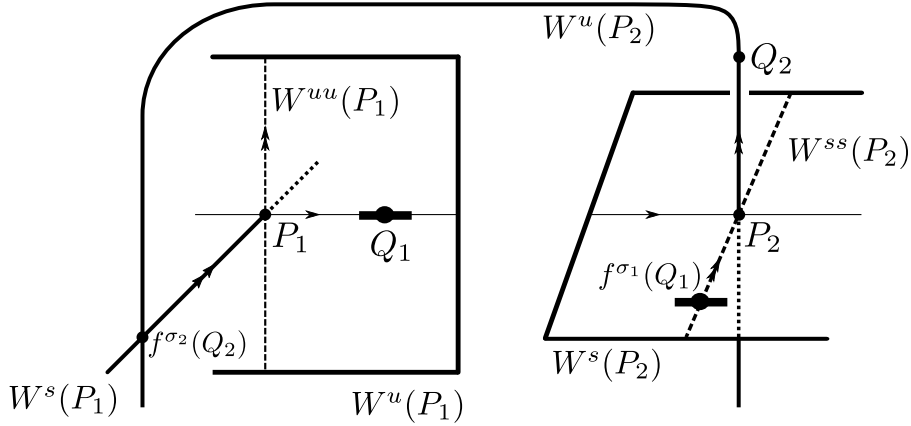}
\caption{An illustration of an SH-simple cycle. }
\label{fig:SH}
\end{center}
\end{figure}

\begin{rema}
Let $f$ be a diffeomorphism having a SH-simple cycle.
According to the coordinates of $Q_1$ and $Q_2$, 
we may assume that they have the following forms:
\begin{align*}
\phi_1(Q_1) 
&= (\boldsymbol{0}^{d_s}, q_1, \boldsymbol{0}^{d_u}), &
\phi_2(f^{\sigma_1}(Q_1)) 
&= (\boldsymbol{q}'_1, 0, \boldsymbol{0}^{d_u}). \\
\phi_2(Q_2) 
&= (\boldsymbol{0}^{d_s}, 0, \boldsymbol{q}_2), &
\phi_1(f^{\sigma_2}(Q_2)) 
&= (\boldsymbol{q}_2', 0, \boldsymbol{0}^{d_u}).
\end{align*}
\end{rema}

\subsection{Main perturbation result}

In \cite{BD}, it was proven that given a 
diffeomorphism having a heterodimensional 
cycle, by adding an arbitrarily small $C^1$-perturbation 
one can obtain another diffeomorphism such that the 
continuation of the heterodimensional cycle is simple, 
that is, the local dynamics around it is given by 
locally affine maps.

One of the main step of the proof of our theorem 
is that one can obtain similar affine dynamics from 
strongly heteroclinic cycles. 

\begin{prop}\label{prop:SHsim}
Let $f\in\mathrm{Diff}^1(M)$ which 
has two hyperbolic fixed points $P_1$, $P_2$ and
satisfies all the assumptions (T1-4) in Theorem~\ref{theo:main}. Then, there exists $g\in\mathrm{Diff}^1(M)$ arbitrarily 
$C^1$-close to $f$ 
such that the continuations of $P_1$ and $P_2$ form a 
SH-simple cycle for $g$.
\end{prop}

The proof of Proposition~\ref{prop:SHsim} is 
given in Section~\ref{s:pert}.

\subsection{Main analytic result}
Let us state the main analytic result about the 
existence of the periodic points. We prepare one definition. 
For a hyperbolic periodic point $P$ with index 
$\mathbf{d}$, its \emph{center Lyapunov exponent},
denoted by $\lambda_{c}(P)$, 
is the real number given as follows:
\[
\lambda_{c}(P) 
:= \frac{1}{\pi} 
\log \|Df^{\pi}|_{E^c(P)}\|,
\] 
where $\pi$ denotes the period of $P$ and
$E^c(P)$ denotes the center direction at $P$. It is easy to see that periodic points in the same orbit shares the same Lyapunov exponents, thus sometimes we also say Lyapunov exponents of some orbit.  
Usually, the notion of Lyapunov exponents are defined 
for invariant measures. 
Notice that this definition coincides with the usual one 
if we consider the uniformly distributed Dirac measure 
along the orbit of $P$. Below, for a diffeomorphisms $f \in \mathrm{Diff}^1(M)$ and a point $x \in M$, we put 
$\mathrm{orb}(P) := \{ f^i(x) \mid i \in \mathbb{Z} \}$.
\begin{prop}\label{prop:ana}
Let $f \in \mathrm{Diff}^1(M)$ with a heterodimensional cycle
associated to hyperbolic fixed points $P_1$ and $P_2$. 
Suppose that they form a SH-simple heterodimensional cycle 
with respect to the coordinates $(U_i, \phi_i)$.
Then, 
there exists an integer $\tilde{l}$ such that the following holds: 
\begin{itemize}
\item For every $j \geq \tilde{l}$, there exists a periodic point $R_j$ 
of period $j$
and whose orbit admits a strongly 
partially hyperbolic splitting of index $\mathbf{d}$ such 
that the angles between $E^s(R_j), E^c(R_j)$ and 
$E^u(R_j)$ are bounded from below by some positive uniform constant 
independent of $j$. 
\item Let $\lambda_c(R_j)$ be the 
center Lyapunov exponent of $R_j$. 
Then we have $\lambda_c(R_j) \to 0$ as $j \to \infty$. 
\item For every $j \geq \tilde{l}$, the sequence of the orbits $\{\mathrm{orb}(R_k)\}_{k > j}$ 
does not accumulate to $\mathrm{orb}(R_j)$. In other words, for every $j \geq \tilde{l}$, there exists 
a neighbourhood $V_j$ of $\mathrm{orb}(R_j)$ such that 
$V_j \cap \mathrm{orb}(R_k) = \emptyset$ for every 
$k > j$.
\end{itemize}
\end{prop}

\subsection{Proof of Theorem~\ref{theo:main}}
Using Proposition~\ref{prop:SHsim} and Proposition~\ref{prop:ana}, 
let us complete the proof of Theorem~\ref{theo:main}.
For the proof, we prepare a lemma. 
This is used to perturb a periodic point 
having small Lyapunov exponent into one 
with zero Lyapunov exponent.
Notice that 
the availability of this lemma heavily depends 
on the flexibility of the $C^1$ topology. 

\begin{lemm}[Franks' Lemma, see Appendix A of \cite{BDV}]\label{l:franks}
Let $f \in \mathrm{Diff}^1(M)$, $\varepsilon > 0$ 
and $P$ be a hyperbolic periodic point of period $\pi$. 
Let $\{G_i:T_{f^i(P)}M \to T_{f^{i+1}(P)}M\}_{i=0,\ldots, \pi-1}$ be 
a sequence of linear maps such that 
$\|Df_{f^i(P)} -G_i\| < \varepsilon$ holds for every 
$i$. Then, given a neighbourhood $V$ 
of $\mathrm{orb}(P)$, 
there exists $g \in \mathrm{Diff}^1(M)$
such that the following holds: 
\begin{itemize}
\item $\mathrm{dist}(f, g) < \varepsilon$;
\item $g(x)=f(x)$ for every $x\in M\setminus V$;
\item $g$ preserves the orbit of $P$, that is, 
for every $i$ we have $g^i(P) = f^i(P)$;
\item $Dg_{g^i(P)} = G_i$.
\end{itemize}
\end{lemm}

Now, let us complete the proof.

\begin{proof}[Proof of Theorem~\ref{theo:main}]
Let $f \in \mathrm{Diff}^1(M)$
having a strongly partially hyperbolic heterodimensional cycles associated to 
$P_1$ and $P_2$ satisfying the assumptions (T1-4) of Theorem~\ref{theo:main}. 
Fix an arbitrarily small $\varepsilon_{\ast} >0$.

Let us apply Proposition~\ref{prop:SHsim} to this heterodimensional cycle.
Then we obtain a diffeomorphim $f_1$ with
a SH-simple heterodimensional 
cycle associated to the hyperbolic continuations 
of $P_1$ and $P_2$ for $f_1$. 
Notice that $f_1$ 
can be chosen arbitrarily $C^1$-close to $f$, in particular,
$\varepsilon_{\ast}/3$-close to $f$ in the $C^1$-distance.

Now we can apply Proposition~\ref{prop:ana}: For $f_1$ we know that
there exist $\tilde{l} \in \mathbb{N}$ and a sequence of periodic orbits 
$\{ R_j\}_{j \geq \tilde{l}}$
satisfying the conclusion of Proposition~\ref{prop:ana}. 
For each $j \geq \tilde{l}$, we take a neighbourhood $V_j$ of $\mathrm{orb}(R_j)$
in such a way that $V_j \cap V_{j'} = \emptyset$ holds for $j \neq j'$.

Since $\mathrm{orb}(R_j)$ admits a partially hyperbolic splitting and $\lambda_c(R_j) \to 0$ as $j \to \infty$, 
there exists 
$L > \tilde{l}$ such that for every $j \geq L$,
we can find 
an $\varepsilon_{\ast}/3$ $C^1$-small perturbation 
whose the support is contained in $V_j$
such that it preserves the 
orbit of $R_j$ 
and the resulted Lyapunov 
exponent of $R_j$ is equal to zero. 
Furthermore, we can assume that the 
the perturbations are arbitrarily small as $j \to \infty$.

Let us state this more precisely.
For every $j \geq L$, using Franks' Lemma 
we take a $C^1$ diffeomorphism 
$\rho_j\in\mathrm{Diff}^1(M)$ 
such that the following holds:
\begin{itemize}
\item $\mathrm{supp}(\rho_{j}) \subset V_j$ where we put $\mathrm{supp}(\rho_{j}) := \{ x \in M \mid 
\rho_{j}(x) \neq x\}$;
\item For every $k \geq 0$, $ (\rho_j \circ f_1)^k(R_j) = f_1^k(R_j)$.
In particular, $R_j$ is still a periodic point of period $j$ for $\rho_j\circ f_1$.
\item $\lambda_c(R_j, \rho_j \circ f_1) =0$, where $\lambda_c(R_j, \rho_j \circ f_1)$ 
denotes the center Lyapunov exponent 
of $R_j$ for $\rho_j \circ f_1$.
\item $\mathrm{dist}(\rho_j \circ f_1, f_1)
<\varepsilon_{\ast}/4$ and 
it converges to zero as $j \to \infty$.
\end{itemize}
The existence of such a sequence of diffeomorphisms 
can be confirmed by the fact that 
$\lambda_c(R_j, f_1) \to 0$ and the 
boundedness of the angles of partially hyperbolic 
splittings over $\{\mathrm{orb}(R_j)\}$. 
We define the sequence of diffeomorphisms
$\{g_j\}_{j \geq L}$ inductively as follows: 
\begin{itemize}
\item $g_{L} = \rho_{L} \circ f_1$.
\item $g_{j+1} = \rho_{j+1} \circ g_{j}$ 
for $j > L$.
\end{itemize}
Using the disjointness of the support of 
$\{\rho_j\}$, we can see that for every $k \geq 0$,
the sequence
$\mathrm{dist}(g_{j+k},g_j) $ converges to zero as 
$j \to \infty$, uniformly with respect to $k$.
Consequently, $\{g_{j}\}_{j\ge L}$ is a Cauchy sequence in $\mathrm{Diff}^1(M)$. By the completeness of $\mathrm{Diff}^1(M)$ (see \cite{Hi} for instance), the sequence 
$\{ g_j \}$ converges to a 
$C^1$ diffeomorphism in the $C^1$-distance.
Let $g_{\infty}$ be the limit diffeomorphism.
Notice that, again by the disjointness of $V_j$, 
for every $j$ we 
see that the orbit of $R_j$ is the same for $f_1$ 
and $g_{\infty}$, 
having zero center Lyapunov exponent. 
Furthermore, by the continuity of the distance function we have
$\mbox{dist}(f_1, g_{\infty}) \leq 
\varepsilon_{\ast}/4 < \varepsilon_{\ast}/3$.

Let us give the final perturbation to obtain the conclusion.
Since the orbits of $\{ R_j \}$ are the same for 
$g_{\infty}$ and $f_1$, still $\{V_j\}$ are pairwise disjoint 
neighbourhoods of $\mathrm{orb}(R_j)$. 
For each $j \geq L$, we take a diffeomorphism 
$\eta_j\in\mathrm{Diff}^1(M)$ such that
\begin{itemize}
\item $\mathrm{supp}(\eta_{j}) \subset V_j$.
\item $\eta_j \circ g_{\infty}$ has $j \cdot a_j$ 
distinct periodic points of period $j$ in $V_j$, where $a_j$ is some super exponentially divergent sequence 
(for instance set $a_j = j!$).
\item $\mathrm{dist}(\eta_j \circ g_{\infty}, g_{\infty}) < \varepsilon_{\ast}/4$ for every $j \geq L$ and 
converges to zero as $j \to \infty$. 
\end{itemize}
The existence of such $\{\eta_j\}$ can be deduced by using the 
nullity of the center Lyapunov exponent of $R_j$,
see for instance Remark~5.2 in \cite{AST} for the 
concrete construction of such perturbations. 

Then, put 
$h_n := \eta_n \circ \cdots \circ \eta_{L} \circ g_{\infty}$.
By the same reason as above, one can 
check that the limit $h_{\infty} := \lim_{n \to \infty} h_n$ exists and 
it is a $C^1$-diffeomorphism. Furthermore, one can see that $h_{\infty}$
has at least $j\cdot a_j$ periodic points of period $j$ for every $j \geq L$ and
$\mathrm{dist}(g_{\infty}, h_{\infty}) < \varepsilon_{\ast}/3$. 
Finally, we have 
\[
\mathrm{dist}(f, h_{\infty}) \leq \mathrm{dist}(f, f_1) +\mathrm{dist}(f_1, g_{\infty}) + \mathrm{dist}(g_{\infty}, h_{\infty}) < \varepsilon_{\ast}
\]and for every $r>1$, 
\[
\liminf_{n\to\infty}\dfrac{\#\mbox{Per}_n(h_{\infty})}{r^n}\ge \liminf_{n\to\infty}\dfrac{na_n}{r^n}=+\infty.
\]
Thus, the diffeomorphism 
$h_{\infty}$ satisfies the conclusion of Theorem~\ref{theo:main}.
\end{proof}


\section{Creation of strong heterodimensional cycles}
\label{s:creSH}
In this section, we prove Proposition~\ref{prop:pert}
and Theorem~\ref{theo:hclass}. 
In the proof, we use the following 
powerful perturbation lemma by Hayashi \cite{Ha} 
which allows us to create a cycle by connecting invariant
 manifolds of different saddles under 
 a small $C^1$ perturbation.

\begin{lemm}[Connecting Lemma]\label{lemm:cnt}
Let $a_f$ and $b_f$ be a pair of saddles of $f\in\mathrm{Diff}^1(M)$ such that thee are sequences of points $\{y_n\}$ and of natural numbers $\{k_n\}$ satisfying:
\begin{itemize}
\item $y_n\to y\in W^u(a_f)\ (n\to\infty), \ y\not= a_f$; and
\item $f^{k_n}(y_n)\to z\in W^s(b_f)\ (n\to\infty),\ z\not=b_f$.  
\end{itemize}
Then, there is a diffeomorphism $g$ arbitrarily $C^1$ close to $f$ such that $W^u(a_g)$ and $W^s(b_g)$ have a non-empty intersection arbitrarily close to $y$, where $a_g$ (resp. $b_g$) is the hyperbolic continuation of $a_f$ (resp. $b_f$) for $g$.
\end{lemm}

\subsection{Proof of Theorem~\ref{theo:hclass}}
\label{ss:thm3}
We begin with the proof of Theorem~\ref{theo:hclass}. 
Let us recall one general result on the transitivity of the 
systems:
\begin{lemm}[\cite{BG}, Page 32, Proposition 2.2.2.]
\label{lemm:bidense}
Let $X$ be a compact metric space without isolated point 
and $f:X \to X$ is a transitive homeomorphism. 
Put $\mathrm{orb}^{+}(x) := \{f^i(x) \mid i \geq 0\}$
and call it the forward orbit of $x$.
Then there is a residual subset $R \subset X$ such that 
for every $x \in R$, 
$\mathrm{orb}^{+}(x)$ is dense in $X$.
\end{lemm}

Let us give the proof of Theorem~\ref{theo:hclass}.
Notice that almost the same argument appears 
for instance in \cite[Lemma~2.8]{ABCDW}.

\begin{proof}[Proof of Theorem~\ref{theo:hclass}]
Fix an arbitrarily small $\varepsilon >0$. 

First, we fix fundamental domains of 
$W^s(P_1)$ and $W^u(P_2)$ and
denote their closures by $K_1$ and $K_2$ respectively.
Notice that they are compact sets.  
Then, by the transitivity of $f$ on $H(P_1)$ and hyperbolicity 
near $P_1$ and $P_2$, we can choose the sequences of 
orbits $\{y_n\}$ and integers 
$\{k_n\}$ satisfying the assumption 
of the Connecting Lemma, letting $a_f = P_2$ and $b_f=P_1$. 
That is, first we choose 
a point $x \in H(P_1)$ whose forward 
orbit is dense in $H(P_1)$ (see Lemma~\ref{lemm:bidense},
notice that $H(P_1)$ has no isolated point, since it 
is non-trivial). 
Then, by using the hyperbolicity of $P_1$ and $P_2$, 
we can see that $\mathrm{orb}^{+}(x)$ has accumlating 
points in $K_1$ and $K_2$. Then, let $y$ be 
one of the accumlating point in $K_2$ and $z$ be 
one in $K_1$. Then the constructions of 
$\{y_n\}$ and $\{k_n\}$ are straightforward.

Now, by applying Hayashi's Connecting Lemma, 
we obtain an $\varepsilon/2$-small 
$C^1$-perturbation $g$ of $f$ such that $W^s(P^g_1)\cap W^u(P^g_2)\not=\emptyset$, 
where $P^g_i\ (i=1,2)$ denote 
the hyperbolic continuation of $P_i$ for $g$. 

Notice that the transversal intersection of $W^u(P^g_1)$ 
and $W^{ss}(P^g_2)$ is $C^1$-robust. 
Thus $P_1^g$ and $P_2^g$ form 
a heterodimensional cycle that satisfies 
the hypothesis (T1-4) of Theorem~\ref{theo:main} whose conclusion gives a $\varepsilon/2$-small $C^1$-perturbation $h$ of $g$ such that $h$ is super exponential divergent. Since 
$\mathrm{dist}(h,f)\le \mathrm{dist}(h,g)+\mathrm{dist}(g,f)<\varepsilon $ and $\varepsilon$ can be chosen arbitrarily small in advance, we obtain the conclusion of Theorem~\ref{theo:hclass}.
\end{proof}

\subsection{Proof of Proposition~\ref{prop:pert}}

Let us give the proof of Proposition~\ref{prop:pert}.
The proof is divided into two steps.

\begin{lemm}\label{lemm:1gene}
Let $M$ be a three dimensional closed 
manifold. Let $\mathcal{U}$ be an open set of 
$\mathrm{Diff}^1(M)$ such that 
every $f \in \mathcal{U}$ satisfies all the 
conditions in Theorem~\ref{theo:trans}.
Then, there is a set $\mathcal{V} \subset \mathcal{U}$
which is open and dense in $\mathcal{U}$ such that for every
$g \in \mathcal{V}$ either 
$W^u(P_1) \cap W^{ss}(P_2) \neq \emptyset$
or 
$W^{uu}(P_1) \cap W^s(P_2) \neq \emptyset$
holds.
\end{lemm}

\begin{proof}
By the robust transitivity, and the 
preservation of the orientation, we know that 
we can approximate the diffeomorphism by one such that
either the strong 
stable foliation or the strong unstable foliation is minimal
in $M$ (i.e., every leaf is dense in $M$)
see \cite[Theorem~1.3]{BDU}. 
For such a diffeomorphism, we have either
$W^u(P_1) \cap W^{ss}(P_2) \neq \emptyset$
or 
$W^{uu}(P_1) \cap W^s(P_2) \neq \emptyset$.
Since this is an open condition, we obtain the conclusion.
\end{proof}

\begin{lemm}
Let $f \in \mathcal{V}$ in Lemma~\ref{lemm:1gene}.
Then, $f$ can be approximated by $g$ which is 
arbitrarily $C^1$ close to $f$ such that $g$ or $g^{-1}$ satisfies 
conditions (T1-4) in Theorem~\ref{theo:main}.
\end{lemm}

\begin{proof}
Let us take $f \in \mathcal{V}$. 
We assume $W^u(P_1) \cap W^{ss}(P_2) \neq \emptyset$.
The other case can be done by similarly. 
Since $f \in \mathcal{U}$, by Hayashi's connecting lemma 
we can perturb $f$ so that $W^s(P_1) \cap W^u(P_2) \neq \emptyset$ (see the argument in Section~\ref{ss:thm3} 
for the detail). Thus we can obtain (T4) by an arbitrarily small
perturbation. Since the other conditions (T1-3) 
are all $C^1$-robust, 
we have that $g$ satisfies all the conditions (T1-4). 
\end{proof}

\section{Perturbation to SH-simple cycles }\label{s:pert}
In this section, we prove Proposition~\ref{prop:SHsim}. 
The strategy of the proof is close to the proof of Proposition~3.5 of \cite{BD}, which is based on Lemma~3.2 of \cite{BDPR}. We remind the reader that Proposition~\ref{prop:SHsim} is stated for diffeomorphisms of closed manifold of dimension large than or equal to three.   

\begin{proof}[Proof of Proposition~\ref{prop:SHsim}]
Let $f$ be a $C^1$ diffeomorphism with two hyperbolic fixed points $P_1$ and $P_2$ that satisfy the assumptions (T1-4) in Theorem 3. We will construct an arbitrarily small $C^1$ perturbation $g$ of $f$ such that $g$ exhibits a SH-simple cycle associated to $P_1$ and $P_2$. In fact, such a perturbation will be obtained by finitely many steps and the $C^1$ size of the perturbation can be controlled arbitrarily small in each step. 
Let us fix an arbitrarily small $\varepsilon>0$.

\vbox{}
\underline{STEP 1} First,    
we  perturb $f$ in such a way that 
the perturbed diffeomorphism acts as 
an affine map in a small 
neighbourhood of the fixed points and the heteroclinic points. 
Namely, 
by using Franks' Lemma
(see Lemma~\ref{l:franks}) near $P_1, P_2$ and
the heteroclinic points, we take
a diffeomorphism $f_1$ with $\mathrm{dist}(f_1,f)<\dfrac{\varepsilon}{4}$
and local charts $(U_i,\phi_i)$ of $P_i\ (i=1,2)$ such that 
the following holds:
\begin{itemize}
\item $\phi_i(U_i)=\mathbb{D}_{r^i_s}^{d_s}\times \mathbb{D}_{r^i_c}^{1}\times\mathbb{D}_{r^i_u}^{d_u}\subset \mathbb{R}^{\boldsymbol{\mathrm{d}}}$;
\item $\phi_i(P_i)=(\boldsymbol{0}^{r_s},0,\boldsymbol{0}^{r_u})$;
\item $\phi_i\circ f_1\circ \phi_i^{-1}(x_s,x_c,x_u)=(Df|_{E^s(P_i)}(x_s), Df|_{E^c(P_i)}(x_c),Df|_{E^u(P_i)}(x_u))$ for every $(x_s,x_c,x_u)\in \phi_i(U_i)$,
where $E^c$ is the one-dimensional invariant subspace of $T_{P_1}M$ ($T_{P_2}M$) in which $Df(P_1)$ (resp. $Df(P_2)$) has weakest expanding (resp. weakest contracting) eigenvalue. Here, we remind the reader to recall the assumption (T2) in Theorem 3. $E^s(P_i)$ is the $d_s$ dimensional invariant subspace of $Df(P_i)$ associated to its first $d_s$ strongest contracting eigenvalues and $E^u(P_i)$ is the $d_u$ dimensional invariant subspace of $Df(P_i)$ associated to its first $d_u$ strongest expanding eigenvalues. 
\item There exist $\sigma_i\in\mathbb{N}$ and $Q_i\in W^u(P_i)\cap W^s(P_{i+1})$ such that  $Q_i\in U_i$, $f_1^k(Q_1)\notin U_1\cup U_2$ for $k=1,2,\cdots, \sigma_i-1$ and $f_1^{\sigma_i}(Q_i)\in U_{i+1}$, 
where we set $P_3=P_1$ and $(U_3, \phi_3)=(U_1,\phi_1)$. 
In the following we refer the integer $\sigma_i$ as the \emph{first enter time} of $Q_i$ into $U_{i+1}$; 
\item $\phi_2\circ f_1^{\sigma_1}\circ \phi_1^{-1}$ and $\phi_1\circ f_1^{\sigma_2}\circ \phi_2^{-1}$ are affine in a small neighbourhood of $Q_1$ and $Q_2$;
\item We also require the following: Let $t_1$ 
(resp. $t_2$) denote the 
center unstable (resp. stable) eigenvalue of 
$Df_1|_{P_1}$ (resp. $Df_2|_{P_2}$).
Then we have that $\log t_1$ and $\log t_2$ are 
rationally independent. 
\end{itemize}
Such a perturbation can be done similarly 
as in \cite[Lemma~3.2]{BDPR}. 

Inside $U_i$ for $i=1, 2$, 
there are locally invariant foliations which are parallel 
to the coordinate plane. We define as follows:
\begin{itemize}
\item $\mathcal{F}_i^s$ is the $d_s$-dimensional foliations 
on $U_i$ ($i=1, 2$)
with leaves parallel to $\mathbb{D}^{d_s}$ in the local 
coordinate, 
called the strong stable foliation;
\item Similarly, we define foliations 
$\mathcal{F}_i^u$, $\mathcal{F}_i^c$, 
$\mathcal{F}_i^{cs}$
and $\mathcal{F}_i^{cu}$ and call them the strong unstable,
center, center stable, center unstable foliations respectively. 
\end{itemize}

In the following, for $\ast\in\{s, u, c, cs, cu\}$,
by ${(\mathcal{F}^{\ast}_i)}_x$ we mean the leaf of the 
foliation $\mathcal{F}_i^{\ast}$ passing $x$ and 
by $W_{\mathrm{loc}}^{\ast}(P_i)$ we mean the connected 
component of $W^{\ast}(P_i) \cap U_i$ containing $P_i$. 
For instance, $W_{\mathrm{loc}}^u(P_1)$ denotes the 
connected component of 
$U_1\cap W^u(P_1)$ containing $P_1$.
Then, we have the following:
\begin{itemize}
\item inside $U_1$, 
we have $(\mathcal{F}_1^{cu})_{P_1}
=W_{\mathrm{loc}}^u(P_1)$ 
and $(\mathcal{F}_1^{s})_{P_1}=
W_{\mathrm{loc}}^s(P_1)$; 
\item inside $U_2$, we have 
$(\mathcal{F}_{2}^{cs})_{P_2}
=W_{\mathrm{loc}}^s(P_2)$ 
and $(\mathcal{F}_{2}^u)_{P_2}=W_{\mathrm{loc}}^u(P_2)$.
\end{itemize}

\underline{STEP 2} In this step, we construct 
a perturbation of $f_1$ such that the transition 
point $Q_1$ locates in the center foliation of $P_1$. First, by an arbitrarily small $C^1$ perturbation, we can always assume that $Q_1\notin (\mathcal{F}_{1}^u)_{P_1}$. 
By the domination of $E^c(P_1)\oplus E^u(P_1)$, 
we have
$$\dfrac{d_u(\phi_1(f_1^{-k}(Q_1)),\phi_1(f_1^{-(k+1)}(Q_1)))}{d_c(\phi_1(f_1^{-k}(Q_1)),\phi_1(f_1^{-(k+1))}(Q_1))}\to 0\quad (k\to +\infty),$$  
where $d_u(A,B)$ (resp. $d_c(A,B)$) denotes the distance 
between the points $A, B$ along the $\mathcal{F}^u$ (resp. $\mathcal{F}^c$) direction. 
We take $k\in \mathbb{N}$ sufficiently large such that  
$$\dfrac{d_u(\phi_1(f_1^{-k}(Q_1)),\phi_1(f_1^{-(k+1)}(Q_1)))}{d_c(\phi_1(f_1^{-k}(Q_1)),\phi_1(f_1^{-(k+1))}(Q_1))}<\dfrac{\varepsilon}{10M_0},$$ 
where $$M_0=\sup \{\|D(\phi_1 g \phi_1^{-1})(\boldsymbol{p})\|+\|D(\phi_1 g^{-1} \phi_1^{-1})(\boldsymbol{p})\|:\ \boldsymbol{p}\in \phi_1(U_1),\ \mathrm{dist}(g,f_1)<1 \}<+\infty.$$
Then, we take a diffeomorphism, 
denoted by $\alpha$, such that 
\begin{itemize}
\item $\mathrm{dist}(\alpha, \mathrm{id})<\dfrac{\varepsilon}{8M_0}$;
\item $\alpha$ coincides with identity outside a small neighbourhood $U$ of $f_1^{-k}(Q_1)$. Here, $U$ can be taken so small that $U\cap\mathrm{orb}(Q_1)=f_1^{-k}(Q_1)$; 
\item $\alpha\circ f^{-k}(Q_1)\in (\mathcal{F}^c_{1})_{P_1}$.
\end{itemize}
Thus, $f_2=\alpha \circ f_1$ is a perturbation of $f_1$ with 
$\mathrm{dist}(f_2,f_1)<\dfrac{\varepsilon}{4}$ 
satisfying $f_2^{-k}(Q_1)\in (\mathcal{F}^c_{1})_{P_1}$.
Notice that the forward iterations of $Q_1$ are not 
affected by the above perturbation and the 
heterodimensional cycle associated to $P_1$ and $P_2$ also
survives. 
By shrinking $U_1$ and replacing $Q_1$ by some 
backward iteration of it (still denoted by $Q_1$ for 
notational simplicity), we also get a new first enter 
time of $Q_1$ (still denoted by $\sigma_1$) such that 
$Q_1\in U_1\cap (\mathcal{F}^c_{1})_{P_1}$, 
$f_2^{\sigma_1}(Q_1)\in 
(\mathcal{F}_2)^s_{P_2}$ and 
$f_2^{j}(Q_1)\notin U_1\cup U_2$ 
for $j=1,2,\cdots ,\sigma_1-1.$ 
Finally, we fix a small 
neighbourhood $K_1\subset U_1$ of 
$Q_1$ such that $\phi_1(K_1)$ 
is a polydisk. 

\vbox{}
 \underline{STEP 3} Our goal in this step is to get another small perturbation of $f_2$ which keeps the foliations invariant under the transition maps. 
 First, we consider the perturbation around $Q_1$. 
 Without loss of generality, we can assume that $f_2^{\sigma_1}((\mathcal{F}^{cu}_1)_{Q_1})$ 
 is in the general position with respect to 
 $(\mathcal{F}_2^{cu})_{f_2^{\sigma_1}(Q_1)}$. 
 
 By the existence of the domination for 
 $E^s\oplus (E^{c} \oplus E^{u})$, 
 the forward image of 
 $f_2^{\sigma_1}((\mathcal{F}_1^{cu})_{P_1})$ 
under $f_2$ tends to 
$(\mathcal{F}_2^{cu})$. 
Thus, 
by replacing $\sigma_1$ by $\sigma_1+k$ for some large $k$, we can make a small perturbation $f_3$ of $f_2$ 
(again by Franks' lemma)
which keeps the foliation $\mathcal{F}_1^{cu}$ 
invariant under $f^{\sigma_1}_3$ on a smaller $K_1$.
 
 Now we consider the invariance of $\mathcal{F}_{\ast}^s$.
 For the strong stable foliation, 
 we can assume that 
 $f_3^{-\sigma_1}((\mathcal{F}_{2}^s)_{f^{\sigma_1}_3(Q_1)})$ is in a general position with respect to 
 $(\mathcal{F}_1^s)_{Q_1}$. 
 Consider the backward iterations of $f_3^{-\sigma_1}((\mathcal{F}_{2}^s)_{f^{\sigma_1}_3(Q_1)})$, 
which tends to $(\mathcal{F}_1^s)$. 
Replacing $Q_1$ by $f_3^{-k}(Q_1)$ for some large $k$, 
we can take a small perturbation $f_4$ of $f_3$, 
such that the foliation $\mathcal{F}_1^{s}$ 
is invariant under $f^{\sigma_1}_3$ on a smaller $K_1$,
preserving the invariance of center unstable foliations 
$\mathcal{F}_{\ast}^{cu}$ in $K_1$.
 
Repeating the above argument to $\mathcal{F}_{\ast}^{cs}$ and $\mathcal{F}_{\ast}^u$, we obtain small perturbation 
$f_5$ of $f_4$ that $f_5^{\sigma_1}|_{K_1}$ preserves 
the foliations $\mathcal{F}_{\ast}^{cs}$ and $\mathcal{F}_{\ast}^{u}$, 
in addition to $\mathcal{F}_{\ast}^{cu}$ and $\mathcal{F}_{\ast}^{s}$. Then, the preservation of 
$\mathcal{F}_{\ast}^{cs}$ and $\mathcal{F}_{\ast}^{cu}$
implies the preservation of $\mathcal{F}_{\ast}^{c}$. 
Accordingly, we have seen the preservation of all the 
five foliations under $f_5^{\sigma_1}|_{K_1}$.

Completely in a similar way, by an arbitrarily small $C^1$ perturbation, $f_5^{\sigma_2}|_{K_2}$ also preserves these foliations. Each perturbation in this STEP 3 can be made arbitrarily small in the $C^1$ distance, thus we can have $\mathrm{dist}(f_5,f_2)<\dfrac{\varepsilon}{4}.$

\vbox{}
\underline{STEP 4} In this last step, we are going to give the final perturbation $f_6$ of $f_5$ such that 
the transition map 
$f^{\sigma_i}_6|_{K_i}$, 
restricted to the center direction, 
is an isometry (i.e., the has multiplication factor equals to 1). 

Since $\mathcal{F}_{\ast}^c$ is invariant under $f_5$, 
we only need to consider the restriction of 
$f_5^{\sigma_1}$ to $\mathcal{F}_{\ast}^c$,
which has the following form:
$$\ q_1+Y\mapsto bY\quad (b\in \mathbb{R}).$$
For $m,n\in \mathbb{N}$, let us consider $f_5^{-n}(Q_1)$ and $f_5^{\sigma_1+m}(Q_1)$ instead of $Q_1$ and $f_5^{\sigma_1}(Q_1)$ respectively. Since $f_5$ acts as a linear map $Df_5(P_i)$ in $U_i\ (i=1,2)$, we obtain that the restriction of $Df_5^{n+\sigma_1+m}(f^{-n}(Q_1))$ to $\mathcal{F}_{\ast}^c$ is of the following form:
\[\ t_1^{-n}q_1+Y\mapsto t_1^nt_2^mbY,\]
where $t_1\in (0,1)$ and $t_2>1$ are the center multipliers of $Df_5(P_1)$ and $Df_5(P_2)$ respectively. We remind the reader that $f_5$ act as a linear map $Df_5(P_i)$ inside $U_i$ by STEP 1 of our perturbation and 
recall that $\log t_1$ and $\log t_2$ are 
rationally independent. 
Thus we are allowed to choose $n$ and $m$ 
sufficiently large such that
$$|t_1^nt_2^mb-1|<\dfrac{\varepsilon}{4}.$$ 
Consider a linear perturbation $A$ 
of $Df_5(f_5^{\sigma_1+m-1}(Q_1))$ satisfying
\begin{itemize}
\item $A=\mbox{id}$ restricted in $E^s\oplus E^u$ direction;
\item $A=(t_1^nt_2^mb)^{-1}$ restricted in $E^c$ direction;
\item $\|A- \mbox{id}\|<\dfrac{\varepsilon}{4}$.
\end{itemize}
Applying Franks' Lemma to $f_5$ at $f_5^{\sigma_1+m-1}(Q_1)$, we get a $C^1$ perturbation $f_6$ of $f_5$ with $\mathrm{dist}(f_6,f_5)<\dfrac{\varepsilon}{4}$ such that $$Df_6^{\sigma_1+n+m}(f_5^{-n}(Q_1))=A\circ Df_5^{\sigma_1+n+m}(f_5^{-n}(Q_1)).$$
By our construction, the center multiplier of $f_6^{\sigma_1+n+m}(Q_1)$ 
equals to one. Let us rewrite $f_5^{-n}(Q_1)$ by $Q_1$, $\sigma_1+n+m$ by $\sigma_1$ and $f_6$ by $g$, shrink $U_1$ and $U_2$, take small neighbourhood $K_1\subset U_1$ of $Q_1$  such that $\sigma_1$ is the first enter time of $Q_1$ into $U_2$. In a similar way, we give another arbitrarily small $C^1$ perturbation to make the center multiplier of $f^{\sigma_2}(Q_2)$ equal to one. It is easy to verify according to Definition 2.4 that $g$ has a SH-simple cycle associated to $P_1$ and $P_2$. Moreover, 
we have $\mathrm{dist}(g,f)<\varepsilon$, since $\varepsilon$ is taken arbitrarily small in advance, the size of the perturbation can be made arbitrarily small. This completes the proof of Proposition~\ref{prop:SHsim}.    
\end{proof}


\section{Proof of analytic result}

In this section, we prove Proposition~\ref{prop:ana}.
This is an analytic result and the proof is done by purely 
analytic argument.

\subsection{Setting}
In this section, we introduce the maps in the 
local coordinates and gives formal calculations of the 
coordinates of the periodic points around the heterodimensional 
cycle. 

Let us consider a diffeomorphism $f$ 
having an SH-simple cycle between 
the fixed points $P_1$ and $P_2$ with local coordinates 
$(U_i, \phi_i)$ ($i=1, 2$), see Figure.~\ref{fig:SH}.
By definition, we know that for $i=1,2$,
\[
\phi_i(U_i)
=\mathbb{D}_{r_s^i}^{d_s} 
\times \mathbb{D}_{r_c^i}^{1}
\times \mathbb{D}_{r_u^i}^{d_u}.
\]
We put $F_i := \phi_i\circ f \circ \phi_i^{-1}$.
There are linear maps 
$\Lambda_i:  \mathbb{R}^{d_s} \to \mathbb{R}^{d_s}$ and
$M_i:  \mathbb{R}^{d_u} \to \mathbb{R}^{d_u}$
($i=1, 2$), 
$\mu: \mathbb{R} \to \mathbb{R} $, and
$\lambda:  \mathbb{R} \to \mathbb{R}$,
which describe the local dynamics around $P_i$.
In the following, 
we identify $\mu, \lambda$ with a real number which 
gives the multiplications under these linear maps.

Namely, for $i=1$,
for every $(x_s, x_c, x_u) \in \phi_1(U_1)$, 
if $\mu (x_c) \in \mathbb{D}_{r^c_1}^{1}$
and
$M_1 (x_u) \in \mathbb{D}_{r^u_1}^{d_u}$
then we have 
\[
F_1(x_s, x_c, x_u) = (\Lambda_1(x_s), \mu(x_c), M_1(x_u)).
\]
Also, for $i=2$,
for every $(x_s, x_c, x_u) \in \phi_2(U_2)$, 
if $M_2 (x_u) \in \mathbb{D}_{r^u_2}^{d_u}$
then we have 
\[
F_2(x_s, x_c, x_u) = (\Lambda_2(x_s), \lambda(x_c), M_2(x_u)).
\] 

Let us recall the local dynamics around the transition region. 
Let $K_i \subset U_i$ 
be the transition region from $U_i$ to $U_{i+1}$
(we put $U_3=U_1$).
Then $\phi_1(K_1)
= (\boldsymbol{0}^{d_s}, q_{1}, \boldsymbol{0}^{d_u}) + 
\mathbb{D}_{\kappa_s^1}^{d_s} 
\times \mathbb{D}_{\kappa_c^1}^{1}
\times \mathbb{D}_{\kappa_u^1}^{d_u}$.
For $K_2 \subset U_2$, we have
$\phi_2(K_2) 
= (\boldsymbol{0}^{d_s}, 0, \boldsymbol{q_{2}}) + 
\mathbb{D}_{\kappa_s^2}^{d_s} 
\times \mathbb{D}_{\kappa_c^2}^{1}
\times \mathbb{D}_{\kappa_u^2}^{d_u}.$

Let $\sigma_i$ be the transition time from 
$U_i$ to $U_{i+1}$. We put 
$\tilde{F}_i = 
\phi_{i+1}\circ f^{\sigma_i} \circ \phi_i^{-1}$.
Then we have
\[
\tilde{F}_1(\boldsymbol{0}^{d_s}, q_{1}, \boldsymbol{0}^{d_u})
= ( \boldsymbol{q}'_{1}, 0, \boldsymbol{0}^{d_u}), \quad
\tilde{F}_2(\boldsymbol{0}^{d_s}, 0, \boldsymbol{q}_{2})
= (\boldsymbol{q}'_{2}, 0, \boldsymbol{0}^{d_u}).
\]

Again, there are linear maps 
$\tilde{\Lambda}_i:  \mathbb{R}^{d_s} \to \mathbb{R}^{d_s}$ and
$\tilde{M}_i:  \mathbb{R}^{d_u} \to \mathbb{R}^{d_u}$
($i=1, 2$), 
which describes the local dynamics of the transition maps.
That is, for $i=1$,
for every 
$(X, q_{1}+Y, Z)\in K_1$ we have 
(recall the definition of SH-simple cycles, notice that in the $Y$ direction the transition map 
has multiplier $1$)
\[
\tilde{F}_1(X, q_{1}+Y, Z)
= (\boldsymbol{q}'_{1}+ \tilde{\Lambda}_1(X), Y, \tilde{M}_1(Z)).
\]
Similarly, for $i=2$,
for every 
$(X, Y, \boldsymbol{q_{2}}+Z)\in K_2$ we have 
\[
\tilde{F}_2(X, Y, \boldsymbol{q}_{2}+Z)
= (\boldsymbol{q}'_2+\tilde{\Lambda}_2(X), 
Y, \tilde{M}_2(Z)).
\]

\subsection{Formal calculation}
Given a SH-simple cycle, we are interested in finding periodic 
points which turn around it. 
Let us assume that there exists 
a periodic point $R \in K_1$ 
which has the following itinerary:
\begin{itemize}
\item $f^{\sigma_1}(R) \in U_2$;
\item there exists a positive integer $m_2 $ such that
$f^{\sigma_1+k}(R) $ is in the linearized region of $U_2$
for $k=0,\ldots, m_2-1$;
\item $f^{\sigma_1+m_2}(R) \in K_2$ and $f^{\sigma_1+m_2+\sigma_2}(R) \in U_1$;
\item there exists a positive integer $m_1$ such that
$f^{\sigma_1+m_2+\sigma_2+k}(R)$ is 
contained in the linearized region of $U_1$ for 
$k=0,\ldots, m_1-1$;
\item $f^{\sigma_1+m_2+\sigma_2+m_1}(R) =R$.
\end{itemize}
In this subsection, we investigate the condition 
of $R$ in the local coordinates.

Put  
\begin{align}\label{eq:defi}
\phi_1(R) = (X, q_1+Y, Z).
\end{align}
Then, 
$f^{\sigma_1}(R) \in U_2$ has the following form
in the $(U_2, \phi_2)$ coordinates:
\begin{align*}
(\boldsymbol{q}'_{1}+ \tilde{\Lambda}_1(X), Y, \tilde{M}_1(Z)).
\end{align*}
The point $f^{\sigma_1+m_2}(R)$ spends $m_2$ times
in $U_2$. 
As a result, in the $(U_2, \phi_2)$ coordinates, this point has 
the following coordinates:
\[
( \Lambda_2^{m_2}[\boldsymbol{q}'_1+\tilde{\Lambda}_1(X)], 
\lambda^{m_2}Y, M_2^{m_2}\tilde{M}_1(Z))
= 
( \Lambda_2^{m_2}[\boldsymbol{q}'_1+\tilde{\Lambda}_1(X)], 
\lambda^{m_2}Y, 
[M_2^{m_2}\tilde{M}_1(Z) -\boldsymbol{q}_2] +\boldsymbol{q}_2).
\]
Then, under $f^{\sigma_2}$, this point is mapped to 
$f^{\sigma_1+m_2+\sigma_2}(R)$. 
The local coordinates of this point with respect to 
$(U_1, \phi_1)$ is
\[
( \boldsymbol{q}'_2+\tilde{\Lambda}_2\Lambda_2^{m_2}\tilde{\Lambda}_1[\boldsymbol{q}'_1+\tilde{\Lambda}_1 (X)] , 
\lambda^{m_2}Y, 
\tilde{M}_2[M_2^{m_2}\tilde{M}_1(Z)-\boldsymbol{q}_2]).
\]
Then, under $f^{m_1}$, this point is mapped back to $R$.
The local coordinate is equal to
\[
\bigg( \Lambda_1^{m_1}(\boldsymbol{q}'_2+\tilde{\Lambda}_2\Lambda_2^{m_2}\tilde{\Lambda}_1[\boldsymbol{q}'_1+\tilde{\Lambda}_1 (X)]), 
{\mu}^{m_1}\lambda^{m_2}Y, 
M_1^{m_1}\tilde{M}_2[M_2^{m_2}\tilde{M}_1(Z)-\boldsymbol{q}_2] \bigg).
\]
Since this point is equal to the point in (\ref{eq:defi}),
we have equations for $(X, Y, Z)$. Formally, the solution is
\begin{align}
X&=[I -\Lambda_1^{m_1}\tilde{\Lambda}_2\Lambda_2^{m_2}\tilde{\Lambda}^2_1]^{-1}\cdot(
\Lambda_1^{m_1}\tilde{\Lambda}_2\Lambda_2^{m_2}\tilde{\Lambda}_1\boldsymbol{q}'_1 + {\Lambda_1}^{m_1} \boldsymbol{q}'_2), \label{eq:solx}\\
Y&=\frac{q_1}{\mu^{m_1}\lambda^{m_2}-1}, \label{eq:soly} \\
Z&=[M_1^{m_1}\tilde{M}_2M_2^{m_2}\tilde{M}_1-I ]^{-1}
 M_1^{m_1}\tilde{M}_2 \boldsymbol{q}_2 \nonumber \\
&=\big(M_2^{m_2}\tilde{M}_1-\tilde{M}_2^{-1}M_1^{-m_1}\big)^{-1} \boldsymbol{q}_2, \label{eq:solz}
\end{align}
where $I$ denotes the identity map. 
This formal solution may give a true periodic point 
of period $(\sigma_1 +m_2+\sigma_2 + m_1)$ 
depending on the choice of $m_2$ and $m_1$. 
In the following, we consider for what choice of 
$m_2$ and $m_1$ we can obtain the true orbit.

\subsection{Realizability of the orbit}
In order to check that the formal solution obtained 
in the previous subsection gives a true solution, 
we need to confirm that the point indeed passes the 
transition region at the designated moments. 
The following proposition states that
we can judge it by looking the
behavior in the center direction.

\begin{prop}\label{solu}
There exists $M>0$ such that 
for every $(m_1, m_2)$ satisfying $m_1, m_2 \geq M$ 
the following holds:
Suppose that $\mu^{m_1}\lambda^{m_2} \neq 1$ 
and we have the following inequality:
\begin{align}\label{eq:exis}
\left|\frac{q_1}{\mu^{m_1}\lambda^{m_2}-1}\right| \leq \kappa^1_c,
\end{align}
Then the orbit of the point $\bar{R} \in U_1$ given by
 $\phi_1(\bar{R}) = (X, q_1+Y, Z)$, where $X, Y, Z$ are the 
formal solutions (\ref{eq:solx}, \ref{eq:soly}, \ref{eq:solz}),
gives a true periodic orbit. 
\end{prop}

\begin{proof}
First, we can see that the point
$(X, q_1+Y, Z)$ is in the transition region 
$K_1$ if 
both $m_1$ and $m_2$ are sufficiently large. 
Indeed, first $X \to \boldsymbol{0}^s$ as 
$m_1, m_2 \to +\infty$. This is because 
the linear map $\Lambda_1^{m_1}\tilde{\Lambda}_2\Lambda_2^{m_2}\tilde{\Lambda}_1$
goes to zero map and the point $\Lambda_1^{m_1}\tilde{\Lambda}_2\Lambda_2^{m_2}\tilde{\Lambda}_1\boldsymbol{q}'_1 + {\Lambda_1}^{m_1} \boldsymbol{q}'_2$
goes to $\boldsymbol{0}^s$.
Similarly, by \eqref{eq:solz} we have $Z \to \boldsymbol{0}^u$ as 
$m_1, m_2 \to +\infty$.
The inequality \eqref{eq:exis} we assume guarantees that the $Y$-coordinate
of $\phi_1(\bar{R})$ lies the region of $\phi_1(K_1)$. 
Thus the point $\bar{R}$ is indeed in $K_1$
for sufficiently large $m_1$ and $m_2$.
 
\bigskip

Let us confirm that $f^{m_1+\sigma_1}(\bar{R})$ is in $K_2$.
The condition for $X$, $Y$-coordinates are obvious
for larger $m_1$ and $m_2$.
So let us examine the condition 
of the $Z$-coordinate.

By the definition of $Z$ in (\ref{eq:solz}), 
it satisfies 
\[
Z =M_1^{m_1}\tilde{M}_2 \big(M_2^{m_2}\tilde{M}_1Z -\boldsymbol{q}_2\big).
\]
As we have observed, $Z$ is very close to $\boldsymbol{0^u}$ when $m_1, m_2$
are large. Thus $M_2^{m_2}\tilde{M}_1Z -q_2$ must 
be close to zero since it is equal to 
$ (M_1^{m_1}\tilde{M}_2)^{-1}Z$,
where $(M_1^{m_1}\tilde{M}_2)^{-1}$ are 
strongly contracting linear map for larger $m_1$. 
This means that
$ (M_1^{m_1}\tilde{M}_2)^{-1}Z$, 
which is 
the $Z$-coordinate of 
$f^{m_1+\sigma_1}(\bar{R})$ in the local coordinates,
converges to $\boldsymbol{q}_2$ as $m_1, m_2 \to \infty$. 

Thus, we have seen that for $m_1$, $m_2$ large, the 
itinerary of $\bar{R}$ certainly passes the transition 
regions with the given itinerary. This completes the proof.
\end{proof}

\begin{rema}
This proof shows that there is no restriction of the 
orientation of strong stable/unstable eigenvalues of the 
fixed points.
\end{rema}

\subsection{Proof of Proposition~\ref{prop:ana}}
In this subsection we will complete the proof of 
Proposition~\ref{prop:ana} by examining the 
inequality (\ref{eq:exis}). 

By the argument of the previous subsections, 
we know that for $(m_1, m_2)$ sufficiently large
there exists a periodic point of period $\sigma_1 +m_2+
\sigma_2+ m_1$ if and only if it satisfies the 
inequality (\ref{eq:exis}). 
We shall show that for every $l\in\mathbb{N}$, 
there are integers $(m_{1,l}, m_{2,l})$ such that there 
is a periodic point $R_l$ of period 
$\sigma_1 +m_{2,l}+\sigma_2+ m_{1,l}:=\pi(R_l)$, satisfying $\pi(R_{l+1})=\pi(R_l)+1.$ 
whose central Lyapunov exponent 
$\lambda_{c}(R_l)$ converges to zero as $l \to \infty$.

First, by a direct calculation, 
we can get a sufficient condition 
for the inequality (\ref{eq:exis}):
\begin{lemm}\label{alpha}
For fixed $q_1$ and $\kappa_c$, under the 
condition $\lambda^{m_1}\mu^{m_2} > 1$, 
We have the inequality (\ref{eq:exis}) if  
$\lambda^{m_1}\mu^{m_2} > \tilde{\alpha}$, where
\[
\tilde{\alpha}:= \frac{|q_1|}{\kappa^c_1} +1,
\]
Notice that $\tilde{\alpha} > 1$.
\end{lemm}
Now we are ready to complete the proof.

\begin{proof}[End of the proof of Proposition~\ref{prop:ana}]
Let $C := \max\{|\log\lambda|, |\log\mu|\}$ and 
choose $L, L' > \log\tilde{\alpha}$ such that 
$L' -L> 2C$ holds. Then, we investigate the pair of integers 
$(m_1, m_2)$ satisfying 
\[
L < m_2 \log \lambda + m_1 \log \mu < L'.
\]

Now, we fix some sufficiently large $(m_{1,1}, m_{2,1})$ which 
satisfy the above inequality. Such pair of integers exist since 
$L'-L >2C $. Then we can inductively construct another
pair of integers $(m_{1,2},m_{2,2})$ which also satisfies above inequality 
and $m_{1,2}+m_{2,2} = m_{1,1}+m_{2,1} +1$ holds. 
Indeed, given $(m_{1,1}, m_{2,1})$, by the condition 
$L'-L >2C $ we can see that at least one of
$(m_{1,1}+1, m_{2,1})$ and $(m_{1,1}, m_{2,1}+1)$ satisfies 
the inequality. Let us denote that pair by $(m_{1,2}, m_{2,2})$.

Thus, by induction, we can choose the sequence of the pair of integers $\{(m_{1,l}, m_{2,l})\}_{l=1}^{+\infty}$ such that 
\begin{align}\label{plusone}
\sigma_1+m_{1,l+1}+\sigma_2+m_{2,l+1}=\sigma_1+m_{1,l}+\sigma_2+m_{2,l}+1,
\end{align}
and satisfying the inequality $$L < m_{2,l} \log \lambda + m_{1,l} \log \mu < L'$$ for every $l\in\mathbb{N}$. 
We claim that $R_l:=\phi_1^{-1}\bigl((X_l,q_1+Y_l,Z_l)\bigr)$, where $(X_l,q_1+Y_l,Z_l)$ is the solution of Proposition~\ref{solu} (see (\ref{eq:solx}, \ref{eq:soly}, \ref{eq:solz}))
corresponding to $(m_{1,l},m_{2,l})$, gives the desired sequence of the periodic points. In fact, on the one hand, we have $m_{2,l} \log \lambda + m_{1,l} \log \mu>L>\log\tilde{\alpha}$, thus by Lemma~\ref{alpha} and Proposition~\ref{solu}, there indeed exists a periodic point $R_l$ of period $\pi(R_l)=\sigma_1 +m_{2,l}+\sigma_2+ m_{1,l}$. Moreover, $\pi(R_l)$ satisfies $\pi(R_{l+1})=\pi(R_l)+1$ according to (\ref{plusone}). 
On the other hand, the central 
Lyapunov exponent of $R_l$ is given by
\[
\frac{m_{2,l} \log \lambda + m_{1,l} \log \mu}{\sigma_1+m_{2,l}+\sigma_1+m_{1,l}},
\]  
whose numerator has absolute value bounded by $L'$ from above. Thus 
as $l$ tends to infinity, $(m_{1,l}+m_{2,l})$ goes to infinity as well, which leads to the conclusion that the central Lyapunov exponent converges to zero. Moreover, by translating the subscript of $R_l$, we can make $\pi(R_l)=l$ for every sufficiently large $l\in\mathbb{N}$. 

Let us confirm that there is no self accumulation of 
the sequence of the points $\{R_l\}$. Indeed, by construction
one can check that the accumulation points 
of $\{R_l\}$ are contained in the set 
\[
\phi_1^{-1}(
\{\boldsymbol{0}^{d_s}\}
\times \mathbb{D}_{\kappa_c}^{}
\times 
\{\boldsymbol{0}^{d_u}\})
\]
and it does not contain any point of 
$\{R_l\}$. This implies the 
conclusion. 

Furthermore, by construction 
we know that every $R_j$ admits 
partially hyperbolic 
splitting with bounded angles, 
deriving from the SH-simple cycles
(indeed, the splitting is orthogonal in the 
local coordinate). 

Thus the proof is completed.
\end{proof}


\section{On the generic non-divergence}
In this section, we provide the proof of Threorem~\ref{theo:resi}
which says that Theorem~\ref{theo:main}
cannot be improved to the generic setting. 
We thank Masayuki Asaoka for pointing out the importance 
of the result of \cite{Ka}.

A sequence $(a_n)$ of positive integers is said to
\emph{grows super exponentially} if for 
every $r>1$ we have 
$\lim_{n \to \infty}  r^{n}/a_n \to 0 $  
holds.

The following result by Kaloshin \cite{Ka} is the 
main ingredient of the proof.
\begin{prop}\label{prop:kalo}
Given $1 \leq s < \infty$,
there exists a dense subset $\mathcal{D}^s \subset \mathrm{Diff}^s(M)$ such that
for every $f \in \mathcal{D}^s$ the followings hold:
\begin{itemize} 
\item Every periodic point of $f$ is hyperbolic.
\item There exists a positive real number $C_f >0$ such 
that $\#\mathrm{Per}_n(f) < \exp (C_fn) $ holds for every $n \in \mathbb{N}$.
\end{itemize}
\end{prop}
\begin{rema}
The proof of Proposition~\ref{prop:kalo} does not work for $s =\infty$.
We do not know whether the result is true for $s =\infty$.
Nonetheless, 
we can prove that Theorem~\ref{theo:resi} is true for 
$s =\infty$.
\end{rema}

Let us complete the proof of Theorem~\ref{theo:resi}.
In the following, for every $1 \leq s \leq \infty$ 
we fix some distance function $\mathrm{dist}_{C^s}$ 
which is compatible with the $C^s$-topology. 

\begin{proof}[Proof of Theorem~\ref{theo:resi}]
First, let us consider the case $1\leq s < \infty$.
Given a positive integers $L$ and $s $, we put
\[
\mathcal{O}^{s}_L := \{ f \in \mathrm{Diff}^s(M) \mid \exists N >L, \quad \#\mathrm{Per}^h_N(f) < a_N/N  \},
\]
where $\mathrm{Per}^h_N(f)$ denotes the set of hyperbolic periodic points of $f$ of period $N$.
One can easily see that $\mathcal{O}^{s}_L$ is an open set in $\mathrm{Diff}^s(M)$
with respect to the $C^s$-topology.
Furthermore, by Kaloshin's result, one can see that $\mathcal{O}^{s}_L$ is dense in $\mathrm{Diff}^s(M)$
with respect to the $C^s$-topology.
Now, put $\mathcal{R} = \cap_{L\geq 1} \mathcal{O}^{s}_L$. This is a residual subset in 
$\mathrm{Diff}^s(M)$ and it is straightforward to see that every diffeomorphism in 
$\mathcal{R}$ satisfies the conclusion of the Theorem~\ref{theo:resi}.

\medskip

Now, let us consider the case $s =\infty$. We define the set of diffeomorphisms
$\mathcal{O}^{\infty}_L$ as in the previous case. 
The openness of $\mathcal{O}^{\infty}_L$ in 
$\mathrm{Diff}^{\infty}(M)$ is 
obvious. Let us prove the density of $\mathcal{O}^{\infty}_L$  in $\mathrm{Diff}^{\infty}(M)$.

Given $f \in \mathrm{Diff}^{\infty}(M)$ and $\varepsilon >0$, we only need to show there is a $k\in \mathcal{O}_L^\infty$ with $\mbox{dist}_{C^\infty}(f,k)<\varepsilon$. 
We choose a positive 
integer $t$ such that for $f$ and $g$ in 
$\mathrm{Diff}^{\infty}(M)$
satisfying $\mbox{dist}_{C^t}(f, g) < \varepsilon/2$ the inequality  
$\mbox{dist}_{C^{\infty}}(f, g) < \varepsilon$ holds. 
Now we choose $h \in \mathcal{D}^t$ such that $\mbox{dist}_{C^t}(f, h) < \varepsilon /5$ holds. 
Notice that $h \in \mathcal{O}^{t}_L$.
Now, by the density of $C^{\infty}$ 
diffeomorphisms in $\mathrm{Diff}^t(M)$, 
we choose $k \in \mathrm{Diff}^{\infty}(M)$ such that $k \in \mathcal{O}^{\infty}_L \subset \mathcal{O}^{t}_L $
and $\mbox{dist}_{C^{t}}(h, k) < \varepsilon/5$ hold. 
Now, we have $\mbox{dist}_{C^{t}}(f, k) < \varepsilon /2$ and hence 
$\mbox{dist}_{C^{\infty}}(f, k) < \varepsilon$. Thus have the density of $\mathcal{O}^{\infty}_L$. 

Finally, arguing in the same way as in the case of $s <\infty$, we complete the proof of Theorem~\ref{theo:resi}.
\end{proof}


\vspace{2cm}

\begin{itemize}
\item[]  \emph{Xiaolong Li \quad} (lixl@hust.edu.cn)
\begin{itemize}
\item[] Graduate School of Mathematics and Statistics, 
\item[] Huazhong University of Science and Technology, Luoyu Road 1037, Wuhan, China
\end{itemize}
\item[] \emph{Katsutoshi Shinohara \quad} (ka.shinohara@r.hit-u.ac.jp)
\begin{itemize}
\item[] Graduate School of Business 
Administration,
\item[] Hitotsubashi University, 2-1 Naka, Kunitachi, Tokyo, Japan
\end{itemize}
\end{itemize}


\begin{thebibliography}{99}


\bibitem[ABCDW]{ABCDW} F. Abdenur, C. Bonatti, S. Crovisier, L.J. D\'{i}az and L. Wen,
\emph{Periodic points and homoclinic classes.} 
 Erg. Th. Dyn. Sys., 27 (2007), no. 1, 1--22. 
 
\bibitem[AM]{AM} M. Artin and B. Mazur,
\emph{On periodic points.} 
Ann. Math. Second Series, 81 (1965), no. 1, 82--99.



\bibitem[AST]{AST} M. Asaoka, K. Shinohara and D. Turaev,
 {\it Degenerate behavior in non-hyperbolic semi-group actions
 on the interval: fast growth of periodic points
 and universal dynamics}, 
 Math. Ann., \textbf{368} (2017) No.3--4, 1277--1309.

\bibitem[AST2]{AST2} M. Asaoka, K. Shinohara and D. Turaev,
 {\it Fast growth of the number of periodic points arising from heterodimensional connections}, 
 arXiv:1808.07218v1.



\bibitem[BC]{BC}C. Bonatti and S. Crovisier,
\emph{R\'ecurrence et g\'en\'ericit\'e.}  Invent. Math., 158 (2004), no. 1, 33--104.

\bibitem[BD]{BD} C. Bonatti
and L.J. D\'{i}az,
 \emph{Robuts heterodimensional cycles and $C^1$-generic dynamics.}
J. Inst. Math. Jussieu, 7 (2008), no.3, 469--525. 

\bibitem[BDF]{BDF}
C. Bonatti, L.J. D\'{i}az and T. Fisher, 
\emph{Super-exponential growth of 
the number of periodic orbits inside homoclinic classes.}
Discrete Contin. Dyn. Syst., 20 (2008), No.3, 589--604.

\bibitem[BDU]{BDU} C. Bonatti, L. J. D\'{i}az and R. Ures,
{\it Minimality of strong stable and 
unstable foliations for 
partially hyperbolic diffeomorphisms},
Journal of the Inst. of Math. Jussieu {\bf 1} (2002) No.4, 
513--541.

\bibitem[BDV]{BDV} C. Bonatti, L. J. D\'{i}az and M. Viana,
{\it Dynamics Beyond Uniform Hyperbolicity},
Springer (2004).

\bibitem[BDPR]{BDPR} C. Bonatti,  L.J. D\'{i}az, E.R. Pujals
and J. Rocha, 
\emph{Robust transitivity and heterodimensional cycles}, 
Ast\'{e}risque, 286 (2003) 187--222. 

\bibitem[Be]{Be} 
P. Berger,
{\it Generic family displaying robustly a fast growth of 
the number of periodic points}, preprint, arXiv:1701.02393.


\bibitem[BG]{BG}
M.Brin and G.Stuck, 
\emph{Introduction to Dynamical Systems}, 
Cambridge University Press (2002).

\bibitem[Ha]{Ha} S. Hayashi, 
\emph{Connecting Invariant manifolds and the solution of the $C^1$ stability and $\Omega$-stability conjecture for flows.} 
 Ann. of Math, Volume 145, (1997), 81--137.
 
\bibitem[Hi]{Hi} M. Hirsch, 
\emph{Differential Topology.} 
 Springer, Graduate Texts in Mathematics, (1976). 
 
 
\bibitem[Ka]{Ka} V. Kaloshin, 
\emph{An extension of the
Artin-Mazur theorem.} 
 Ann. Math, Volume 150, (1999), 729--741.
 



\end{thebibliography}
\end{document}